\documentclass[11pt]{article}
\usepackage{ amssymb }
\usepackage{lipsum}

\usepackage{amsfonts}
\usepackage{amsthm}

\usepackage{epsfig}
\usepackage{graphicx}  
  
\usepackage{pdfsync}

\usepackage[utf8]{inputenc}

\usepackage{caption}
\usepackage{subcaption}
 
\usepackage{amsmath,amsfonts,amssymb,amsthm,epsfig,epstopdf,titling,url,array}

\usepackage{hyperref}

\theoremstyle{plain}
\newtheorem{thm}{Theorem}[section]
\newtheorem{lem}[thm]{Lemma}
\newtheorem{prop}[thm]{Proposition}
\newtheorem{cor}{Corollary}

\theoremstyle{definition}
\newtheorem{defn}{Definition}[section]

\theoremstyle{remark}

\theoremstyle{remark}

\usepackage{graphicx}
\usepackage{xcolor}
\definecolor{darkblue}{rgb}{0,0,0.5} 
\usepackage{transparent}

\begin{document}

\title{A new upper bound for the Asymptotic Dimension of RACGs.}

\author{Panagiotis Tselekidis}


\maketitle



\begin{abstract} 
Let $W_\Gamma$ be the Right-Angled Coxeter group with defining graph $\Gamma$. We show that the asymptotic dimension of $W_\Gamma$ is smaller than or equal to $dim_{CC}(\Gamma)$, the clique-connected dimension of the graph. 
As a corollary we show that $W_\Gamma$ is virtually free if and only if $dim_{CC}(\Gamma)=1$.
\end{abstract}

\tableofcontents

\section{Introduction}

Coxeter groups touch upon a number of areas of mathematics, such as representation theory, combinatorics, topology, and geometry. They are often considered as a playground for many open problems in Geometric Group Theory.

It is known by an isometric embedding theorem of Januszkiewicz (see \cite{TJ}) that Coxeter groups have finite asymptotic dimension. In particular, Januszkiewicz's theorem shows that for any Coxeter group $W_\Gamma$ with defining graph $\Gamma$ we have the following upper bound: $asdimW_\Gamma \leq \sharp V(\Gamma)$.\\
A lower bound for the asymptotic dimension of Coxeter groups was given by Dranishnikov in \cite{Dra.Cohom}, $vcd(W_\Gamma) \leq asdim W_\Gamma$.

Right-angled Coxeter groups (RACGs) are the simplest examples of Coxeter groups, in these, the only relations between distinct generators are commuting relations. In other words, RACGs are the Coxeter groups defined by RAAGs. Dranishnikov proved (see \cite{Dra08}) that the asymptotic dimension of RACGs is bounded from above by the dimension of their Davis complex.\\
\\
\textbf{Question.}
\emph{Is possible to determine the asympotitic dimension of a RACG from its defining graph?}\\

In many cases Dranishnikov's bound is far from being optimal. For example, if the defining graph $\Gamma$ is a clique with $n$ vertices, then by Dranishnikov's result we have that $asdim W_{\Gamma}\leq n$, however, $asdim W_{\Gamma}=0$.\\
The aim of this paper is to provide a new upper bound for the asymptotic dimension of RACGs treating some of the cases in which Dranishnikov's bound fails to be optimal. The main result of this article and its corollaries make some progress towards the previous question. 

We prove the following:

\begin{thm}\label{RACGthm1.1}
Let $W_\Gamma$ be the Right-Angled Coxeter group with connected defining graph $\Gamma$. Then
$$asdim W_\Gamma \leq dim_{CC} (\Gamma).$$
If $\Gamma$ is not connected, then $asdim W_\Gamma \leq max \lbrace 1, dim_{CC} (\Gamma) \rbrace .$
\end{thm}

The clique-connected dimension of a finite graph, $dim_{CC}(\Gamma)$ (see subsection \ref{subsubsection3.1.1}), can be described as an index showing how much connected is the graph modulo cliques.
For example, if $\Gamma$ is a clique, then $dim_{CC}(\Gamma)=0$. In the case of $\Gamma$ being a clique, we have that $W_\Gamma$ is finite, so $asdim W_\Gamma = dim_{CC}(\Gamma)$.  
We will further show that if $dim_{CC}(\Gamma)\leq 2$, then $asdim W_\Gamma = dim_{CC}(\Gamma)$ (see proposition \ref{RACG5.31}).

Since there are cases where the dimension of the Davis complex $\Sigma(W_\Gamma)$ is smaller than $dim_{CC}(\Gamma)$ and other cases where $dim_{CC}(\Gamma)< dim \Sigma(W_\Gamma)$, we have as a corollary:

\begin{thm}\label{RACGthm1.2}
Let $W_\Gamma$ be the Right-Angled Coxeter group with defining graph $\Gamma$. Then
$$asdim W_\Gamma \leq min \lbrace dim_{CC} (\Gamma), dim\Sigma(W_\Gamma) \rbrace.$$
If $\Gamma$ is not connected then $asdim W_\Gamma \leq  max \lbrace 1, min \lbrace dim_{CC} (\Gamma) , dim\Sigma(W_\Gamma) \rbrace  \rbrace.$

\end{thm}

As a corollary of theorem \ref{RACGthm1.1} we prove that: 
\begin{prop}\label{RACG1.3}
Let $W_\Gamma$ be the Right-Angled Coxeter group with connected defining graph $\Gamma$. If $dim_{CC}(\Gamma)\leq 2$, then $asdim W_\Gamma = dim_{CC}(\Gamma)$.
\end{prop}

Interestingly, we can deduce whether $W_\Gamma$ is virtually free from the clique-connected dimension of its defining graph. In particular, we show the following:

\begin{prop}\label{RACG1.4}
Let $W_\Gamma$ be the Right-Angled Coxeter group with connected defining graph $\Gamma$. Then $W_\Gamma$ is virtually free if and only if $dim_{CC}(\Gamma)=1$.
\end{prop}

The paper is organized as follows. In subsection 3.1 we start with some basic definitions and some preliminary results that are used in the rest of the paper. Subsection 3.2 is containing some important lemmas, for example, we show that $dim_{CC}(\ast)$ is ''monotone'' in the following sense: if $\Gamma^{\prime}$ is a full-subgraphs of $\Gamma$, then $dim_{CC}(\Gamma^{\prime}) \leq dim_{CC}(\Gamma)$. In subsection 3.3, we prove that the clique-connected dimension is increasing in some cases. In subsection 3.4, we prove the main theorem of the paper.
In the last subsection, we present some corollaries of the main theorem.\\
\\
\textbf{Acknowledgments:} I would like to thank Panos Papasoglu for his valuable advices during the development of this research work.
 
\section{Preliminaries.}
The asymptotic dimension $asdimX$ of a metric space $X$ is defined as follows: $asdimX \leq n$ if and only if for every $R > 0$ there exists a uniformly bounded covering $\mathcal{U}$ of $X$ such that the R-multiplicity of $\mathcal{U}$ is smaller than or equal to $n+1$ (i.e. every R-ball in $X$ intersects at most $n+1$ elements of $\mathcal{U}$).\\
There are many equivalent ways to define the asymptotic dimension of a metric space. It turns out that the asymptotic dimension of an infinite tree is $1$ and the asymptotic dimension of $\mathbb{E}^{n}$ is $n$.
\begin{flushleft}
By finite simplicial labeled graph we mean a finite simplicial graph $\Gamma$ such that every 
edge $[a,b]$ is labeled by a natural number $m_{ab}>1$. The Coxeter group
associated to $\Gamma$ is the group $W_\Gamma$ given by the following presentation:
\end{flushleft}

\begin{center}
$W_\Gamma= \langle V(\Gamma) \vert a^2 $ for all $a \in V(\Gamma)$ and $  \underbrace{ab \ldots}_{m_{ab}}= \underbrace{ba \ldots}_{m_{ab}} $ when a, b are connected by an edge $ \rangle $.

\end{center}
A Coxeter group is called\textit{ Right-angled Coxeter group} (RACG) if $m_{ab}=2$ when a, b are connected by an edge.\\

We say a simplicial graph is \emph{complete} or equivalently \emph{a clique} if any two vetrices are connected by an edge. An \emph{n-clique} is the complete graph on n vertices.\\
We recall that the \emph{full subgraph} defined by a subset $V$ of the vertices of a graph $\Gamma$ is a subgraph of $\Gamma$ formed from $V$ and from all of the edges that have both endpoints in the subset $V$.
If $G$ is a subgraph of $\Gamma$, we denote by $FS^{\Gamma}(G)$ the full-subgraph of $\Gamma$ defined by $V(G)$.\\
The \emph{simplicial closure} of $\Gamma$ is the flag complex $SC(\Gamma)$ defined by $\Gamma$.\\

We recall the definition of a \textit{parabolic subgroup} of a Coxeter group. Let $\Gamma$ be a finite simplicial labeled graph and $W_\Gamma$ be the Coxeter group associated to $\Gamma$. Let $X$ be a subset of $V(\Gamma)$, we denote by $\Gamma_X$ the full subgraph of $\Gamma$ formed from $X$, and by $G_{X}$ the subgroup of $W_\Gamma$ generated by $X$ (we see $X$ as a subset of the natural generating set of $W_\Gamma$). We consider the graph $\Gamma_X$ as a labeled graph inheriting its labeling from $\Gamma$.\\
It is known that $G_X$ is a Coxeter group, it is actually equal to $W_{\Gamma_{X}}$, the Coxeter group associated to $\Gamma_{X}$ see(\cite{EM}).\\
The subgroup $G_X=W_{\Gamma_{X}}$ is called \emph{standard parabolic subgroup} of $W_\Gamma$.\\

The following theorem is proved by Dranishnikov in \cite{Dra08}.

\begin{thm}\label{RACGT2.0}
For any finitely generated groups $A$ and $B$ with a common finitely generated subgroup $C$ we have:
\begin{center}
$asdim A \ast_C B \leq max \lbrace asdim A, asdim B, asdim C + 1 \rbrace.$
\end{center}
\end{thm}

The following theorem is a generalisazion of theorem \ref{RACGT2.0}. It was proved by the author in \cite{PT}.
\begin{thm}\label{RACGT2.1}
Let $(\mathbb{G}, Y)$ be a finite graph of groups with vertex groups $\lbrace G_{v} \mid v \in Y^{0} \rbrace$ and edge groups $\lbrace G_{e} \mid e \in Y^{1}_{+} \rbrace$. Then the following inequality holds:
\begin{center}
$asdim\,\pi_{1}(\mathbb{G},Y,\mathbb{T})  \leq max_{v \in Y^{0} ,e \in Y^{1}_{+}} \lbrace asdim G_{v}, asdim\,G_{e} +1 \rbrace.$
\end{center}

\end{thm}

We also need a theorem for free products from \cite{BDK}, see also \cite{BD04}.
\begin{thm}\label{RACGT2.2}
Let $A$, $B$ be two finitely generated groups. Then:
\begin{center}
$asdim\,A \ast B  = max \lbrace asdim A, asdim B, 1 \rbrace.$
\end{center}

\end{thm}

\subsection{$dim_{CC}(\ast)$}\label{subsubsection3.1.1}
Let $\Gamma$ be a connected simplicial. We say that a finite subset $S$ of $V(\Gamma)$ is a \emph{vertex cut} 
of $\Gamma$ if $FS^\Gamma(S)$ separates the graph (i.e. $\Gamma \setminus FS^\Gamma(S)$ contains at least two connected components) and no other subset of $S$ does that.
\begin{defn}\label{RACGdefn2.1}
Let $\Gamma$ be a simplicial graph and let $ \mathcal{C}_{\Gamma}= \lbrace C_1 , \ldots ,C_i , \ldots \rbrace$ be a collection of distinct cliques of $\Gamma$. We say $\mathcal{C}_{\Gamma}$ is a \emph{clique twin} of $\Gamma$ if the following conditions are satisfied:\\
\textbf{i)} Each $C_i$ is a maximal clique in $\Gamma$ i.e. there is no other clique in $\Gamma$ containing $C_i$.\\
\textbf{ii)} If $C$ is a clique of $\Gamma$, then there is a clique in $\mathcal{C}_\Gamma$ containing $C$.
\end{defn}
We note that the last condition can be replaced by the following:\\
$\cup_{C_i \in \mathcal{C}_\Gamma} SC(C_i) = SC(\Gamma)$.



\begin{defn}\label{RACGdefn2.2}
Let $\Gamma$ be a simplicial graph which has at least one clique twin $ \mathcal{C}_{\Gamma}$. We set
\begin{center}
$m_C (\Gamma)= min \lbrace \sharp \mathcal{C}_{\Gamma} \vert$ where $ \mathcal{C}_{\Gamma}$ is a clique twin of    $ \Gamma \rbrace.$
\end{center}
\end{defn}


\begin{defn}\label{RACGdefn2.3}
Let $\Gamma$ be a connected simplicial graph. We set
\begin{center}
$CC (\Gamma)= max \lbrace 0, min \lbrace m_C (FS^\Gamma(S)), 0  \vert$ where $S$ is a vertex cut of $ \Gamma \rbrace  \rbrace.$
\end{center}
\end{defn}
The number $CC(\Gamma)$ ''measures'' how much connected is the graph $\Gamma$ modulo its cliques. If $S$ is a vertex cut of $\Gamma$ such that $CC (\Gamma)= m_C (FS^\Gamma(S))$, we say that $S$ is a \emph{minimal vertex cut} of $\Gamma$.\\ 

Observe that we can generalize the previous definition to all simplicial graphs, by setting $CC(\Gamma)= min \lbrace CC(E) \mid E$ is a component of $ \Gamma \rbrace$. Finally, we can define $dim_{CC}(\ast)$. 

\begin{defn}\label{RACGdefn2.4}
Let $\Gamma$ be a simplicial. We set
\begin{center}
$dim_{CC} (\Gamma)= sup \lbrace CC (G) \vert$ where $G$ is a full subgraph of $ \Gamma \rbrace.$
\end{center}
\end{defn}
The clique-connected dimension $dim_{CC}(\Gamma)$ ''measures'' how much connected is the graph $\Gamma$ modulo its cliques by taking into account all the full-subgraphs of $\Gamma$.

\section{Basic Lemmas.}
\begin{lem}{(Existence of clique twins.)}\label{RACGdefn3.1}\\
Let $\Gamma$ be a finite simplicial graph. Then there exists at least one clique twin of $\Gamma$.
\end{lem}
\begin{proof} 
We use induction on the number of vertices of the graph. Obviously, the lemma is true if the graph is just a vertex. We assume the lemma is true for any graph with less than $N+1$ vertices. Let $\Gamma$ be a finite simplicial graph with $N+1$ vertices.\\
If the graph is disconnected the lemma follows by the inductive hypothesis and the fact that the clique twins of the components of the graph forms a clique twin of $\Gamma$.\\
So we assume that the graph is connected. We choose an arbitrary vertex say $v$ and we consider the graph $\Gamma_v = \Gamma \setminus v$. By inductive hypothesis, there exists a clique twin $\mathcal{C}_{\Gamma_v}$ of $\Gamma_v$. Since the graph is connected the link of $v$ in $\Gamma$ is non-empty ($lk_\Gamma(v) \neq \varnothing$ ). We enumerate the vertices of the link, $lk_\Gamma(v)=\lbrace v_1, v_2 , \ldots ,v_k \rbrace$. 
For every $C$ in $\mathcal{C}_{\Gamma_v}$ we set: 
\begin{center}
$\overline{C}= $ either $ C$ (if $C \cup v$ doesn't define a clique of $\Gamma$), or the clique defined by $C$ and $v$ (otherwise).
\end{center} 
For every $v_i$ in $lk_\Gamma(v)$ we define: 
\begin{center}
$\mathcal{C}_i $ to be the collection of all maximal cliques in $\Gamma$ containing both $v$ and $v_i$.
\end{center} 
Observe then that the union $\mathcal{C}_\Gamma = (\cup_i \mathcal{C}_i ) \cup \lbrace \overline{C} \vert C \in  \mathcal{C}_{\Gamma_v} \rbrace$ satisfies the conditions of definition \ref{RACGdefn2.1}, so it is a clique twin of $\Gamma$.

\end{proof}

We will see that every finite graph has a unique clique twin. The proof of the previous lemma actually give us a description of how to construct the clique twin of every graph. 

\begin{lem}\label{RACGdefn3.2}
Let $\Gamma$ be a connected finite simplicial graph. Then $\Gamma$ is a clique if and only if $CC(\Gamma)=0=dim_{CC}(\Gamma)$.
\end{lem}
\begin{proof}
We assume that $\Gamma$ is a clique, then obviously there is no vertex cut of $\Gamma$. So by definition \ref{RACGdefn2.3} we have that $CC(\Gamma)=0$.

We now prove the other direction, so we assume that $CC(\Gamma)=0$. If $\Gamma$ is not a clique there exist two vertices $a,b$ such that they are not connected by an edge, then $V(\Gamma)\setminus \lbrace a,b \rbrace$ separate the graph. Obviously then we may find a vertex cut $S$ of $\Gamma$, so by lemma \ref{RACGdefn3.1} $CC(\Gamma)>0.$ 
\end{proof}

\begin{lem}{(Uniqness of clique twins.)}\label{RACGdefn3.5}\\
Let $\Gamma$ be a finite simplicial graph. Then there exists exactly one clique twin of $\Gamma$.
\end{lem}
\begin{proof} By lemma \ref{RACGdefn3.1} there exists a clique twin $\mathcal{C}_\Gamma$ of $\Gamma$. We assume that there exists another clique twin of the graph, say $\mathcal{C}_\Gamma^\prime$. By condition (ii) of definition of clique twins we observe that for every $C^\prime \in \mathcal{C}_\Gamma^\prime$ there exists a $C \in \mathcal{C}_\Gamma$ containing $C^\prime$. 
By condition (i) we obtain that if $C^\prime \subseteq C$, then $C^\prime=C$. Thus $\mathcal{C}_\Gamma^\prime  \subseteq \mathcal{C}_\Gamma$. Using the same argument we conclude that $\mathcal{C}_\Gamma^\prime  = \mathcal{C}_\Gamma$.


\end{proof}

\begin{lem}{(Monotonicity of $dim_{CC}(\ast)$.)}\label{RACGdefn3.4}\\
Let $\Gamma$ be a simplicial graph and let $G$ be a full subgraph of $\Gamma$. Then 
\begin{center}
$dim_{CC}(G) \leq dim_{CC}(\Gamma)$.
\end{center}
\end{lem}
\begin{proof} It suffices to show the lemma when both $G$ and $\Gamma$ are connected. So we assume that $G$ and $\Gamma$ are connected. Since $G$ is a full subgraph of $\Gamma$ we have that every full subgraph of $G$ is also a full subgraph of $\Gamma$. The lemma follows by the definition of $dim_{CC}$.
\end{proof}

If $G$ is not full subgraph the previous lemma is not true.

\begin{lem}\label{RAACGs3.6}
Let $\Gamma$ be a connected simplicial graph such that $CC(\Gamma) \geq 2$. Then $CC(G) \leq 1$ for every $G$ proper full subgraph of $\Gamma$ if and only if $\Gamma$ is a $k$-cycle ($k \geq 4$).

\end{lem}

\begin{proof} Suppose that $CC(G) \leq 1$ for every $G$ proper full subgraph of $\Gamma$. Let $S$ be a vertex cut of $\Gamma$ such that $m_C(S) \geq 2$ and let $\mathcal{C}_S= \lbrace C_1, \ldots , C_k \rbrace$ be the clique twin of $FS^\Gamma (S)$.\\
Let $E_1 , E_2$ be two of the components of $\Gamma \setminus FS^{\Gamma}(S)$. Observe that there are vertices $v_1 , v_2 \in FS^{\Gamma}(S)$ such that they are not connected by an edge in $\Gamma$. Obviously, they belong to distinct cliques. We may assume that $v_1 \in C_1 \setminus C_2$ and $v_2 \in C_2 \setminus C_1$. We note that for every $s\in S$ and every component $E_i$ there exists an edge connecting $s$ with $E_i$ (it follows from the fact that $S$ is a vertex cut of the graph). Thus there exist edge paths $p_i \subseteq E_i \cup C_1 \cup C_2$ connecting $v_1$ with $v_2$ such that $p_i \cap (C_1 \cup C_2) = \lbrace v_1, v_2 \rbrace$. We may assume that these paths have the minimum possible length.\\
Observe that the length of these edge paths is at least two. Trivially, the union $p_1 \cup p_2$ is $k$-cycle where $k \geq 4$.\\
It remains to show that $p_1 \cup p_2$ is a full subgraph of $\Gamma$.
Indeed, it follows from the choice of $v_1$ and $v_2$, the fact that $p_1$, $p_2$ are of minimum length and that $p_1 \setminus (C_1 \cup C_2) $, $p_2 \setminus (C_1 \cup C_2)$ belong to distinct components of $\Gamma \setminus FS^{\Gamma}(S)$.\\
Obviously, $CC(p_1 \cup p_2)=2$. By the hypothesis of lemma we conclude that $\Gamma = p_1 \cup p_2$.

We assume that $\Gamma$ is a $k$-cycle ($k \geq 4$). Let $G$ proper full subgraph of $\Gamma$. Then there is a vertex $v$ of $\Gamma$ such that $G$ is a full subgraph of $\Gamma \setminus v$. 
We observe that if we remove a vertex from $\Gamma$, then the resulting graph $\Gamma^\prime$ is a concatenation  of edges. Trivially, $CC(\Gamma^\prime)=1$ and $CC(G) \leq 1$ for every $G$ full subgraph of $\Gamma^\prime$.

\end{proof}

\section{An increasing property of $dim_{CC}(\ast)$.}
Lemma \ref{RACGdefn3.4} will play a vital role to prove our main theorem but it is not enough for a complete proof, we need something stronger. The main result of this chapter is an interesting increasing property of $dim_{CC}(\ast)$. We will show that $ dim_{CC}(G) < dim_{CC}(\Gamma) $, for some full subgraphs $G$ of $\Gamma$.  



\begin{lem}\label{RACG4.1}
Let $\Gamma$ simplicial graph and let $G$ be a full subgraph of $\Gamma$. Then $m_C (G) \leq m_C (\Gamma)$.
\end{lem}
\begin{proof}
By lemma \ref{RACGdefn3.5}, there exists unique clique twins $\mathcal{C}_G$ of $G$ and $\mathcal{C}_\Gamma$  of $\Gamma$.
By condition (ii) of definition of clique twins we observe that for every $C_G \in \mathcal{C}_G$ there exists a $C_{\Gamma} \in \mathcal{C}_{\Gamma}$ containing $C_G$. In addition to, since $G$ is a full subgraph of $\Gamma$ this correspondence is 1-1 meaning there are no two distinct cliques of $\mathcal{C}_G$ contained in the same clique of $\mathcal{C}_{\Gamma}$. Thus $m_C (G) =\sharp \mathcal{C}_G \leq \sharp \mathcal{C}_\Gamma =  m_C (\Gamma)$.

\end{proof}

\begin{prop}\label{RACG4.2}
Let $\Gamma$ be a finite simplicial graph. Then
$$CC(\Gamma)< m_C (\Gamma).$$
\end{prop}
\begin{proof} It is suffices to show the proposition for connected graphs. If $\Gamma$ is a clique, then the proposition holds. We assume that the graph is not a clique. Then by lemma \ref{RACGdefn3.2} there exists a minimal vertex cut of $\Gamma$, say $S$.\\
By lemma \ref{RACGdefn3.5}, there exists a unique clique twin $\mathcal{C}_S$ of $FS^\Gamma(S)$. We denote by $\mathcal{C}_\Sigma$ the unique clique twin of $\Gamma$. We have $\sharp \mathcal{C}_S = m_C  (FS^\Gamma(S))=CC(\Gamma)$.

By the previous lemma we have that $CC(\Gamma) =\sharp \mathcal{C}_S =  m_C  (FS^\Gamma(S)) \leq    m_C (\Gamma)$.

Let $E_1, \ldots ,E_n$ ($n \geq 2$) be the components of $\Gamma \setminus FS^\Gamma(S)$. Obviously there are no edges connecting $E_i$ with $E_j$ when $i \neq j$. Since $\Gamma$ is connected and $S$ is a vertex cut of the graph we have that for every vertex $s$ of $S$ there exists at least one edge connecting $s$ with $E_i$, for all $i$'s.\\ 
By the above facts, each $C \in \mathcal{C}_S$ is contained in at least two distinct cliques $\Sigma_C^1, \Sigma_C^2$ of $\mathcal{C}_\Sigma$. Since $FS^\Gamma(S)$ is a full subgraph of $\Gamma$, the cliques $\Sigma_C^i, \Sigma_{C^{\prime}}^j$ are distinct when $C \neq C^\prime$, for $i,j \in \lbrace1,2\rbrace$ and $C , C^\prime \in \mathcal{C}_S$. We conclude that the clique twin of $\Gamma$ must contain at least one more clique than $\mathcal{C}_S$. Thus $CC(\Gamma)< m_C (\Gamma)$.

\end{proof}

\begin{prop}\label{RACG4.3}
Let $\Gamma$ be a connected finite simplicial graph and $S$ be a minimal vertex cut of $\Gamma$. Then
$$CC(FS^\Gamma(S))< CC (\Gamma).$$
\end{prop}
\begin{proof}
By proposition \ref{RACG4.2} we have that $CC(FS^\Gamma(S))< m_C(FS^\Gamma(S)) = CC(\Gamma)$.
\end{proof}

The next theorem is the main result of this section.
\begin{thm}\label{RACGthm4}
Let $\Gamma$ be a connected finite simplicial graph and $S$ be a minimal vertex cut of $\Gamma$. Then for every full subgraph $G$ of $FS^\Gamma(S)$ we have the following:
$$dim_{CC}(G)< dim_{CC} (\Gamma).$$
In particular, $dim_{CC}(FS^\Gamma(S))< dim_{CC} (\Gamma).$
\end{thm}
\begin{proof} It is suffices to show the theorem only for connected full subgraphs, so we assume that $G$ is connected.\\
Let $H$ be a full subgraph of $G$, observe that $H$ is a full subgraph of $\Gamma$ as well.
By proposition \ref{RACG4.2} and lemma \ref{RACG4.1} we have that $CC(H)< m_C (H) \leq m_C(G) \leq m_C(FS^\Gamma(S))=CC (\Gamma) $. Thus $CC(H)<CC(\Gamma)$.  

\end{proof}

\section{Asymptotic dimension of RACGs.}

\begin{thm}\label{RACGthm5.1}
Let $W_\Gamma$ be the Right-Angled Coxeter group with connected defining graph $\Gamma$. Then $W_{\Gamma}$ is the fundamental group of a graph of group such that $asdim(G_v) \leq dim_{CC}(\Gamma)$ for every vertex groups and $asdim(G_e) < dim_{CC}(\Gamma)$ for every edge group.
In particular,
$$asdim W_\Gamma \leq dim_{CC} (\Gamma).$$
If $\Gamma$ is not connected, then $asdim W_\Gamma \leq max \lbrace 1, dim_{CC} (\Gamma) \rbrace .$
\end{thm}
\begin{proof} We will use induction on $\sharp V(\Gamma)$. If $\sharp V(\Gamma)=1$, the theorem is obviously true. We assume that for any graph with $\sharp V(\Gamma)< N+1$ the theorem holds. Let $\Gamma$ be a graph such that $\sharp V(\Gamma)= N+1$.

By theorem \ref{RACGT2.2} it is enough to prove the inequality only for RACGs with connected defining graphs. So we assume that the graph is connected. If the graph $\Gamma$ is a clique, then by lemma \ref{RACGdefn3.2} the theorem holds. So we further assume that the graph is not a clique.\\
Since the graph is not a clique there is a subset of its vertices separating it, thus there is at least one vertex cut of $\Gamma$. Let $S$ be a minimal vertex cut of the graph and let $E_1, \ldots , E_k$ be the connected components of $\Gamma \setminus FS^\Gamma(S)$. Observe that since $S$ is a vertex cut we have that for every vertex $s$ of $S$ and every component $E_i$ there exists at least one edge connecting them, thus $E_i \cup FS^\Gamma(S)$ is connected. We further note that $FS^\Gamma (E_i \cup S)= E_i \cup FS^\Gamma (S)$, so $\overline{E}_i = E_i \cup FS^\Gamma(S)$ is a full subgraph of $\Gamma$, and thus by lemma \ref{RACGdefn3.4} $dim_{CC}(\overline{E}_i) \leq dim_{CC}(\Gamma)$. By the inductive hypothesis $asdim W_{\overline{E}_i} \leq dim_{CC}(\overline{E}_i)$, so 
\begin{equation}
asdim W_{\overline{E}_i} \leq dim_{CC}(\Gamma),
\end{equation}

where $ W_{\overline{E}_i} $ is the parabolic subgroup of $W_\Gamma$ defined by $\overline{E}_i$ (of course $W_{\overline{E}_i}$ is RACG). Using theorem \ref{RACGthm4} we have $dim_{CC}(FS^\Gamma (S)) < dim_{CC}(\Gamma)$. By the inductive hypothesis $asdim W_{FS^\Gamma (S)} \leq dim_{CC}(FS^\Gamma (S))$, then
\begin{equation}
asdim W_{FS^\Gamma (S)} < dim_{CC}(\Gamma),
\end{equation}

where $ W_{FS^\Gamma (S)} $ is the parabolic subgroup of $W_\Gamma$ defined by $FS^\Gamma (S)$ (of course $W_{FS^\Gamma (S)}$ is RACG).\\
Finally, observe that $W_\Gamma$ can be obtained from $W_{\overline{E}_i}$ after a finite sequence of amalgamated product over $W_{FS^\Gamma (S)} $. To be more precise, 

\begin{equation}
 W_{\Gamma} = (W_{\overline{E}_1 } \underset{W_{FS^\Gamma (S)}}{\ast} W_{\overline{E}_2}) \underset{W_{FS^\Gamma (S)}}{\ast} \ldots W_{\overline{E}_k}.
\end{equation}

In other words, $W_{\Gamma}$ is the fundamental group of a graph of groups with vertex groups $W_{\overline{E}_i }$ and $W_{FS^\Gamma (S)}$, and edge groups isomorphic to $W_{FS^\Gamma (S)}$. Applying theorem \ref{RACGT2.0} or theorem \ref{RACGT2.1} we conclude that $asdim W_\Gamma \leq dim_{CC} (\Gamma).$

\end{proof}

As a corollary we have: 
\begin{thm}\label{RACGthm1.2}
Let $W_\Gamma$ be the Right-Angled Coxeter group with defining graph $\Gamma$. Then
$$asdim W_\Gamma \leq min \lbrace dim_{CC} (\Gamma), dim\Sigma(W_\Gamma) \rbrace.$$
If $\Gamma$ is not connected then $asdim W_\Gamma \leq  max \lbrace 1, min \lbrace dim_{CC} (\Gamma) , dim\Sigma(W_\Gamma) \rbrace  \rbrace.$
\end{thm}

\section{Corollaries of the main result.}


\begin{prop}\label{RACG5.3} 
Let $\Gamma$ be a simplicial graph such that $dim_{CC}(\Gamma) \geq2$. Then 
the Right-Angled Coxeter group $W_\Gamma$ defined by $\Gamma$ contains an one-ended parabolic subgroup.
\end{prop}
\begin{proof}
$\Gamma$ contains a full subgraph $G$ such that $CC(G) \geq 2$. We assume that $G$ is a minimal full subgraph of $\Gamma$ such that $CC(G) \geq 2$. Trivially, $G$ is connected.
Since $G$ is minimal we have that $CC(G^\prime) \leq 1$ for every $G^\prime$ proper full subgraph of $G$, so by lemma \ref{RAACGs3.6} $G$ is a $k$-cycle ($k \geq 4$).\\
By theorem 8.7.2 of \cite{Davis} we have that the parabolic subgroup $W_G$ of $W_\Gamma$ defined by $G$ is one-ended.
\end{proof}

\begin{prop}\label{RACG5.31}
Let $W_\Gamma$ be the Right-Angled Coxeter group with connected defining graph $\Gamma$. If $dim_{CC}(\Gamma)\leq 2$, then $asdim W_\Gamma = dim_{CC}(\Gamma)$.
\end{prop}
\begin{proof} If $dim_{CC}(\Gamma)=0$, then $\Gamma$ is a clique, so $W_\Gamma$ is finite. Then $asdim W_\Gamma =0$.

If $dim_{CC}(\Gamma)=1$, then by theorem \ref{RACGthm5.1} we have $asdim W_\Gamma \leq 1$. By lemma \ref{RACGdefn3.2} $\Gamma$ is not a clique and thus there are two vertices $a,b$ which are not connected by an edge. This means that $W_\Gamma$ contains $\mathbb{Z}_2 \ast \mathbb{Z}_2$ as a parabolic subgroup, so $asdim W_\Gamma = 1$. 

If $dim_{CC}(\Gamma)=2$, then by theorem \ref{RACGthm5.1} we have $asdim W_\Gamma \leq 2$. By proposition \ref{RACG5.3} we have that there exists an one-ended parabolic subgroup $W_G$ of $W_\Gamma$. Then by the main theorem of \cite{Ge} we obtain that $2 \leq asdim W_G$.
So $asdim W_\Gamma = 2$.


\end{proof}

\begin{cor}\label{RACGCor5}
Let $W_\Gamma$ be the Right-Angled Coxeter group with connected defining graph $\Gamma$. Then $W_\Gamma$ is finite if and only if $dim_{CC}(\Gamma)=0$.
\end{cor}
\begin{proof} Suppose that $W_\Gamma$ is finite. Then $\Gamma$ is a clique, indeed, otherwise $W_\Gamma$ contains $\mathbb{Z}_2 \ast \mathbb{Z}_2$ as a parabolic subgroup, so $asdimW_\Gamma >0$. Which is a contradiction.
Since $\Gamma$ is not a clique by lemma \ref{RACGdefn3.2} we obtain $dim_{CC}(\Gamma)=0$.

The other direction follows by the previous proposition.
\end{proof}

When $\Gamma$ is connected and has clique-connected dimension one, the graph looks like a ''thick'' tree. 

\begin{prop}\label{RACG6.3}
Let $W_\Gamma$ be the Right-Angled Coxeter group with connected defining graph $\Gamma$. Then $W_\Gamma$ is virtually free if and only if $dim_{CC}(\Gamma)=1$.
\end{prop}
\begin{proof} We assume that $W_\Gamma$ is virtually free. If $dim_{CC}(\Gamma) \geq 2$, then by proposition \ref{RACG5.3} $W_\Gamma$ contains an one-ended parabolic subgroup. Since one-ended groups have asymptotic dimension at least two (see \cite{Ge}) we have that $asdimW_\Gamma \geq 2$. By the fact that the asymptotic dimension of virtually free grous is one (see \cite{Ge}) we have a contradiction.\\
If $dim_{CC}(\Gamma) =0$, then $\Gamma$ is a clique. In that case, $W_\Gamma$ is finite, which is a contradiction.

Suppose that $dim_{CC}(\Gamma)=1$, then by proposition \ref{RACG5.31} we have $asdim W_\Gamma = 1$. Applying Gentimis' theorem for virtually free groups (see \cite{Ge}), we conclude that $W_\Gamma$ is virtually free.

\end{proof}

We obtain as a corollary the following:

\begin{prop}\label{RACG6.4}
Let $W_\Gamma$ be the Right-Angled Coxeter group with connected defining graph $\Gamma$. Then $asdim(W_\Gamma) \geq 2$ if and only if $dim_{CC}(\Gamma)\geq2$.
\end{prop}
\begin{proof}We suppose that $asdim(W_\Gamma) \geq 2$, then by theorem \ref{RACGthm5.1} we have that $dim_{CC}(\Gamma)\geq2$.

Conversely, we assume that $dim_{CC}(\Gamma) \geq 2$, then by the previous proposition and the fact that the only groups having asymptotic dimension one are the virtually free groups (see \cite{Ge}) we have that $asdim(W_\Gamma) \neq 1$. Obviously, $asdim(W_\Gamma) \neq 0$, otherwise we have a contradiction by corollary \ref{RACGCor5}.

\end{proof}

Observe that proposition \ref{RACG6.3} and corollary \ref{RACGCor5} can be rephrased as follows:

\begin{center}
Corollary \ref{RACGCor5}: $asdimW_\Gamma=0$ if and only if $dim_{CC}(\Gamma)=0$.
\end{center}

\begin{center}
Proposition \ref{RACG6.3}: $asdimW_\Gamma=1$ if and only if $dim_{CC}(\Gamma)=1$.
\end{center}

We know by proposition \ref{RACG5.31} that if $dim_{CC}(\Gamma)=2$, then $asdimW_\Gamma=2$. One may ask whether the converse is true. We note that by the previous proposition if $asdim(W_\Gamma) = 2$, then $dim_{CC}(\Gamma)\geq2$.\\
\textbf{Question.} \emph{Is there any connected graph such that the RACG defined by the this graph has asymptotic dimension two while the clique connected dimension of the graph is greater than two?}\\

The answer is yes.
We will construct a graph $X$ with clique connected dimension equal to four while $asdimW_X=2$. Let $X_1$, $X_2$ and $X_3$ be 4-cycles with vertices $\lbrace v_1^1,\ldots,v^1_4 \rbrace, \lbrace v_1^2,\ldots,v^2_4\rbrace$ and $\lbrace v_1^3,\ldots,v^3_4 \rbrace$. We join the vertices $v^i_j,v^{i+1}_j$ with edges. The resulting graph $X$ has clique connected dimension equal to four. The graph $X$ is actually the 1-skeleton of a cube complex, thus $Sim(X)=2$.
By Dranishnikov's upper bound (see \cite{Dra08}) and the fact that $W_{X}$ is one ended we obtain that $asdimW_{X}=2$.

Thus an analogue of corollary \ref{RACGCor5} and proposition \ref{RACG6.3} for asymptotic dimension two doesn't exist.
However, we have the following:
\begin{prop}\label{RACG6.5}
Let $W_\Gamma$ be the Right-Angled Coxeter group with connected defining graph $\Gamma$. If $asdimW_\Gamma=2$, then there exists a full subgraph $G$ of $\Gamma$ such that $dim_{CC}(G)=2$ and $asdimW_G =2$.
\end{prop}
\begin{proof}By theorem \ref{RACGthm5.1}, $dim_{CC}(\Gamma)\geq 2$. By the proof of proposition \ref{RACG5.3}, $\Gamma$ contains a $k$-cylce $G$ as a full subgraph ($k \geq 4 $). Trivially, $dim_{CC}(G)=2$.

\end{proof}




\textit{E-mail}: panagiotis.tselekidis@queens.ox.ac.uk


\textit{Address:} Mathematical Institute, University of Oxford, Andrew Wiles Building, Woodstock Rd, Oxford OX2 6GG, U.K.


\end{document}